\documentclass[12pt,a4paper]{amsart}
\allowdisplaybreaks[1]

\usepackage{amstext}
\usepackage{amsthm}
\usepackage{amssymb}

\newtheorem{thm}{Theorem}[section]
\newtheorem{lem}[thm]{Lemma}
\newtheorem{prop}[thm]{Proposition}
\newtheorem{cor}[thm]{Corollary}

\theoremstyle{definition}

\newtheorem{ex}[thm]{Example}
\theoremstyle{remark}
\newtheorem{rem}[thm]{Remark}

\DeclareMathOperator{\sspan}{span}

\begin{document}

\title[linear relations of specific families of MZVs]{$\mathbb{Q}$-linear relations of specific families of multiple zeta values and the linear part of Kawashima's relation}

\author{Minoru Hirose}
\address[Minoru Hirose]{Faculty of Mathematics, Kyushu University 744, Motooka, Nishi-ku,
Fukuoka, 819-0395, Japan}
\email{m-hirose@math.kyushu-u.ac.jp}

\author{Hideki Murahara}
\address[Hideki Murahara]{Nakamura Gakuen University Graduate School, 5-7-1, Befu, Jonan-ku,
Fukuoka, 814-0198, Japan}
\email{hmurahara@nakamura-u.ac.jp}

\author{Tomokazu Onozuka}
\address[Tomokazu Onozuka]{Multiple Zeta Research Center, Kyushu University 744, Motooka, Nishi-ku,
Fukuoka, 819-0395, Japan}
\email{t-onozuka@math.kyushu-u.ac.jp}


\begin{abstract}
 In this paper, we study specific families of multiple zeta values which closely relate to the linear part of Kawashima's relation.
 We obtain an explicit basis of these families, and investigate their interpolations to complex functions.
 As a corollary of our main results, we also see that the duality formula and the derivation relation are deduced from the linear part of Kawashima's relation.
\end{abstract}

\maketitle

\section{Introduction}
 For positive integers $k_1,\ldots,k_r$ with $k_r \ge 2$, the multiple zeta values (MZVs) are defined by
\begin{align*}
 \zeta(k_1,\ldots, k_r)
 :=\sum_{1\le n_1<\cdots <n_r} \frac {1}{n_1^{k_1}\cdots n_r^{k_r}}.
\end{align*}
There exist a lot of $\mathbb{Q}$-linear relations among MZVs,
and Kawashima's relation (with the shuffle product formula) is a big family
of such relations.
One can conjecturally obtain every $\mathbb{Q}$-linear relation among MZVs
from Kawashima's relation
(if the products of MZVs are expanded by the shuffle product formula).
The linear part of Kawashima's relation is a special case of Kawashima's relation.
It does not exhaust all $\mathbb{Q}$-linear relations among MZVs,
but contains several types of well-known relations among MZVs such as the duality formula,
the derivation relation, the quasi-derivation relation, the cyclic sum formula,
the Ohno relation, and so on (see Tanaka \cite[Section 1]{Tan09}).
The purpose of this paper is to study specific families of MZVs
coming from Hoffman's algebraic setup,
and to clarify their close relationship to the linear part of Kawashima's relation.
We obtain an explicit basis of these families by using the linear part of Kawashima's relation,
and give interpolations of these families to complex functions.

We recall Hoffman's algebraic setup with a slightly different convention (see Hoffman \cite{Hof97}).
Let $\mathfrak{H}:=\mathbb{Q} \left\langle x,y \right\rangle$ be the noncommutative polynomial ring in two indeterminates $x$ and $y$.
We define the $\mathbb{Q}$-linear map $\mathit{Z}:y\mathfrak{H}x \to \mathbb{R}$
by $\mathit{Z}( yx^{k_1-1}\cdots yx^{k_r-1}):= \zeta (k_1,\ldots, k_r)$.
For  $A\in y\mathfrak{H}$ and $B\in \mathfrak{H}x$, we define a family $f(A,B)\in \mathbb{R}^{\infty}$ by
\[
 f(A,B):=(\mathit{Z}(AB),\mathit{Z}(A(x+y)B),\mathit{Z}(A(x+y)^2B), \ldots).
\]
Then, the duality formula (Example \ref{duality}) and the derivation relation (Example \ref{der}) for MZVs can be written as linear relations among $f(A,B)$'s.
We define the product $\diamond$ on $\mathfrak{H}$ by
\begin{align*}
 &\qquad\qquad w \diamond 1=1 \diamond w=w, \\
 &xw_1 \diamond xw_2 =x(w_1 \diamond xw_2) - x(yw_1 \diamond w_2), \\
 &xw_1 \diamond yw_2 =x(w_1 \diamond yw_2) + y(xw_1 \diamond w_2), \\
 &yw_1 \diamond xw_2 =y(w_1 \diamond xw_2) + x(yw_1 \diamond w_2), \\
 &yw_1 \diamond yw_2 =y(w_1 \diamond yw_2) - y(xw_1 \diamond w_2)
\end{align*}
for $w,w_1,w_2\in\mathfrak{H}$ together with $\mathbb{Q}$-bilinearity.
\begin{rem}
 We note that the product $\diamond$ is associative and commutative (see Corollary \ref{ac}).
\end{rem}
Let $\tau$ be the anti-automorphism on $\mathfrak{H}$ that interchanges $x$ and $y$.
\begin{thm}[Main theorem] \label{main1} We have the followings.
 \begin{enumerate}
  \item
   For $w_1,w_2\in\mathfrak{H}$, we have
   \[
    f(yw_1, w_2x)=f(y, (\tau(w_1) \diamond w_2)x ).
   \]
  \item
   For $w\in \mathfrak{H}$, we have
   \[
    f(y,wx)=0 \Leftrightarrow w=0.
   \]
  \item
   For $A\in y\mathfrak{H},B\in \mathfrak{H}x$, there exist unique Dirichlet series $L_1(s;A,B)$ and $L_2(s;A,B)$ such that
   \[ f(A,B)=(L_1(s;A,B)+sL_2(s;A,B))_{s=0}^\infty. \]
 \end{enumerate}
\end{thm}
As a corollary of Theorem \ref{main1} (i) and (ii), we can obtain an explicit basis of these families.
\begin{cor}\label{Cor:explicit_basis}
 There exists an explicit basis
 \[ \{ f(y,u_1\cdots u_n x) \mid n\in\mathbb{Z}_{\ge0}, u_1,\ldots,u_n\in\{x,y\} \} \]
 of the $\mathbb{Q}$-linear space spanned by
 \[ \{ f(yw_1,w_2x) \mid w_1,w_2 \in\mathfrak{H} \}. \]
 In addition, all $\mathbb{Q}$-linear relations of $f(-,-)$ are spanned by the relation in Theorem \ref{main1} (i).
\end{cor}
\begin{rem}
 The Dirichlet series $L_1(s;A,B)$ and $L_2(s;A,B)$ can be continued meromorphically to the whole plane $\mathbb{C}$ and explicitly written by using the multiple zeta functions (for details, see Section 3).
\end{rem}
\begin{rem} \label{Rem:equivalence}
 Theorem \ref{main1} (i) is essentially equivalent to the linear part of Kawashima's relation (see Kawashima \cite{Kaw09}).
 This equivalence will be shown in Section 2.
\end{rem}

As a consequence of Corollary \ref{Cor:explicit_basis} and Remark \ref{Rem:equivalence}, we obtain the next theorem.

\begin{thm}[{A special case $A_6 \subset A_1$ of Theorem \ref{Thm:equivalence}}] \label{rem1}
 If
 \[
 \sum_{A,B} n_{A,B} f(A,B)=0,
 \]
 then the relation $\sum_{A,B} n_{A,B} \mathit{Z} (AB)=0$ is deduced from the linear part of Kawashima's  relation.
\end{thm}
%

Let us see some examples of $\mathbb{Q}$-linear relations among $f(-,-)$'s,
and check what Theorem \ref{rem1} says for those examples.
\begin{ex}[Duality formula] \label{duality}
 The duality formula for MZVs is described as $\mathit{Z}(w)=\mathit{Z}(\tau(w)) \, (w \in y\mathfrak{H}x)$.
 From this, for $A\in y\mathfrak{H}$ and $B\in \mathfrak{H}x$, we have
 \[
  f(A,B)=f(\tau(B),\tau(A)).
 \]
 Thus, by Theorem \ref{rem1}, the duality formula is deduced from the linear part of Kawashima's relation.
 This fact is already known (see Kawashima \cite[Section 7]{Kaw09}).
\end{ex}
\begin{ex}[Derivation relation; Ihara-Kaneko-Zagier \cite{IKZ06}] \label{der}
 Let $l$ be a positive integer.
 A derivation $\partial$ on $\mathfrak{H}$ is a $\mathbb{Q}$-linear map $\partial\colon\mathfrak{H}\to\mathfrak{H}$ satisfying Leibniz's rule $\partial(ww^{\prime})=\partial(w)w^{\prime}+w\partial(w^{\prime})$.
 We define the derivation $\partial_l$ on $\mathfrak{H}$ by $\partial_l(x):=y(x+y)^{l-1}x$ and $\partial_l(y):=-y(x+y)^{l-1}x$,
 and we note that $\partial_l(x+y)=0$.
 Then, the derivation relation for MZVs is described as
 \[ \mathit{Z}(\partial_l (w))=0 \quad (w\in y\mathfrak{H}x). \]
 From this i.e.\ $\mathit{Z}(\partial_l (A(x+y)^sB))=0$, we have
 \[
  f(\partial_l (A),B)+f(A,\partial_l (B))=0
 \]
 for $A\in y\mathfrak{H}$ and $B\in \mathfrak{H}x$.
 Thus, by Theorem \ref{rem1}, the derivation relation is
 deduced from the linear part of Kawashima's relation.
 This fact is already known (see Tanaka \cite{Tan09}).
\end{ex}

\section{Proof of Theorem \ref{main1} (i)}
We define an automorphism $\phi$ on $\mathfrak{H}$ by $\phi (x):=x+y$ and $\phi (y):=-y$.
We note that $\phi\circ\phi=\rm{id}$.
For $w, w_1, w_2 \in \mathfrak{H}$, the harmonic product $\ast$ on $\mathfrak{H}$ is defined by
\begin{align*}
 1\ast w &= w \ast 1 = w, \\
 x w_1 \ast w_2 &= w_1 \ast x w_2 = x(w_1 \ast w_2), \\
 y w_1 \ast y w_2 &= y (w_1 \ast y w_2 + y w_1 \ast w_2 + x (w_1 \ast w_2))
\end{align*}
together with $\mathbb{Q}$-bilinearity.
\begin{rem}
 The definition of the above harmonic product $\ast$ coincides with the one given by Sato and the first-named author in \cite{HS18}.
 The harmonic product $\ast$ is associative and commutative.
 We also note that the definition of the usual harmonic product (see Hoffman \cite{Hof97}) is naturally derived from the above definition.
\end{rem}
\begin{thm}[Kawashima's relation; Kawashima \cite{Kaw09}] \label{kawashima_lin}
 For $w_1, w_2 \in y\mathfrak{H}$, we have
 \begin{align*}
  \mathit{Z}(\phi (w_1 \ast w_2)x)=0.
 \end{align*}
\end{thm}
\begin{prop} \label{x+y}
 For $w_1,w_2\in\mathfrak{H}$, we have
 \[
  (x+y)w_1\diamond w_2 = w_1\diamond (x+y)w_2 = (x+y)(w_1 \diamond w_2).
 \]
\end{prop}
\begin{proof}
 The equality $(x+y)w_1\diamond w_2 = (x+y)(w_1\diamond w_2)$ is evident by the definition of the product $\diamond$.
 Then, we prove $(x+y)w_1 \diamond w_2 = w_1\diamond (x+y)w_2$ by induction on the degree of $w_1$.
 Due to the symmetry of the definition, we need to prove only for $w_1$ starting with $x$.
 By the definition and the induction hypothesis, we have
 \begin{align*}
  &xw_1 \diamond (x+y)w_2 \\
  &=x(w_1 \diamond (x+y)w_2) - x(yw_1 \diamond w_2) +y(xw_1 \diamond w_2) \\
  &=x((x+y)w_1 \diamond w_2) - x(yw_1 \diamond w_2) +y(xw_1 \diamond w_2) \\
  &=(x+y)(xw_1 \diamond w_2) \\
  &=(x+y)xw_1 \diamond w_2.
 \end{align*}
 This finishes the proof.
\end{proof}
\begin{lem} \label{harmonic_diamond}
 For $w_1, w_2 \in \mathfrak{H}$, we have
 \begin{align*}
  \phi (w_1 \ast w_2)=\phi(w_1) \diamond \phi(w_2).
 \end{align*}
\end{lem}
\begin{proof}
 We prove by induction on the number of indeterminates of $w_1$ and $w_2$.
 When $w_1=1$ or $w_2=1$, the lemma trivially holds.
 By Proposition \ref{x+y}, the definitions of the products $\ast$ and $\diamond$, and the induction hypothesis, for $u_1,u_2\in\{x,y\}$ and $w_1,w_2\in\mathfrak{H}$, we find
 \begin{align*}
  &\phi(u_1w_1*u_2w_2)\\&=
  \begin{cases}
   \phi(x) (\phi(w_1) \diamond \phi(u_2w_2)) &(u_1=x), \\
   \phi(x) (\phi(yw_1) \diamond \phi(w_2)) &(u_1=y,u_2=x), \\
   \phi(y) (\phi(w_1) \diamond \phi(yw_2) + \phi(yw_1) \diamond \phi(w_2) + \phi(xw_1)\diamond\phi(w_2)) &(u_1=u_2=y)
  \end{cases} \\
  &=\phi(u_1w_1) \diamond \phi(u_2w_2). \qedhere
 \end{align*}
\end{proof}
Since the harmonic product $\ast$ is associative and commutative, we easily see the following corollary holds.
\begin{cor} \label{ac}
 The product $\diamond$ is associative and commutative.
\end{cor}
\begin{lem}[Key lemma for Theorem \ref{main1} (i)] \label{key1}
 The following vector subspaces $A_1,\ldots,A_4$ of $\mathfrak{H}$ are coincident.
 \begin{align*}
  A_1&:= \sspan \{ \phi(u \ast v)x \mid u,v \in y\mathfrak{H} \}, \\
  A_2&:= \sspan \{ (u \diamond v)x \mid u,v \in y\mathfrak{H} \}, \\
  A_3&:= \sspan \{ yw_1(x+y)^n w_2x - y(x+y)^n (\tau(w_1) \diamond w_2)x \mid n\in\mathbb{Z}_{\ge0}, w_1,w_2 \in \mathfrak{H} \}, \\
  A_4&:= \sspan \{ yw_1 w_2x - y (\tau(w_1) \diamond w_2)x \mid w_1,w_2 \in \mathfrak{H} \}.
 \end{align*}
\end{lem}
\begin{proof}
 The inclusion relation $A_4\subset A_3$ is obvious, and $A_1=A_2$ and $A_3\subset A_4$ are trivial respectively by Lemma \ref{harmonic_diamond} and Proposition \ref{x+y}.

 We also find the inclusion relation $A_{2}\subset A_{4}$ holds by the definition of the product $\diamond\,$;
 \begin{align*}
  &(yw_{1}\diamond yw_{2})x \\
  &=y(w_{1}\diamond yw_{2}-xw_{1}\diamond w_{2})x \\
  &=-\left(y\tau(w_{1})(yw_{2})x-y(w_{1}\diamond yw_{2})x\right)+\left(y\tau(xw_{1})w_{2}x-y(xw_{1}\diamond w_{2})x\right).
 \end{align*}
 Thus, we need to show $A_{4}\subset A_{2}$ i.e.,
 \begin{align} \label{eq:A4subA2}
  yw_{1}w_{2}x-y(\tau(w_{1})\diamond w_{2})x\in\sspan\{(u\diamond v)x\mid u,v\in y\mathfrak{H}\}
 \end{align}
 for $w_{1},w_{2}\in\mathfrak{H}$.
 Let
 \[
  w_{1}:=u_{1}\cdots u_{m} \,\, (m\in\mathbb{Z}_{\ge0}, u_{1},\ldots,u_{m}\in\{y,x+y\})
 \]
 and we prove (\ref{eq:A4subA2}) by the induction on $m$.
 The case $m=0$ is obvious.
 Assume that $m>0$.
 Put $w_{1}'u_{m}:=w_{1}$ and $w_{2}':=u_{m}w_{2}$.
 Then we have
 \[
  yw_{1}'w_{2}'x-y(\tau(w_{1}')\diamond w_{2}')x\in\sspan\{(u\diamond v)x\mid u,v\in y\mathfrak{H}\}
 \]
 by the induction hypothesis.
 Thus, we get
 \begin{align*}
  yw_{1}w_{2}x-y(\tau(w_{1})\diamond w_{2})x
  & =yw_{1}'w_{2}'x-y(\tau(w_{1}')\diamond w_{2}')x\\
  &\quad +
  \begin{cases}
   0 & (u_{m}=x+y), \\
   (y\tau(w_{1}') \diamond yw_{2})x & (u_{m}=y)
  \end{cases} \\
  &\in\sspan \{(u\diamond v)x\mid u,v\in y\mathfrak{H}\}. \qedhere
 \end{align*}
\end{proof}
\begin{proof}[Proof of Theorem \ref{main1} (i)]
 By Theorem \ref{kawashima_lin} and Lemma \ref{key1}, we have the desired result.
\end{proof}

\section{Proofs of Theorem \ref{main1} (ii) and (iii)}
To prove Theorem \ref{main1} (iii), we need to show Theorem \ref{main3-1} and \ref{main3-2}.
For positive integers $k_{1},\ldots,k_{r}$,
we denote by $F_{(k_1,\ldots,k_r)}(z)$ the Kawashima function as defined by Yamamoto in \cite{Yam17}.
\begin{lem}[{\cite[Proposition 2.6]{Yam17}}] \label{thm:Fk_at_N}
 For non-negative integer $N$ and positive integers $k_{1},\ldots,k_{r}$, the function $F_{(k_1,\ldots,k_r)}(z)$ is holomorphic at $z=N$ and
 \[
  F_{(k_1,\ldots,k_r)}(N)
  =\sum_{0<n_{1}\le\cdots\le n_{r}\le N}\frac{1}{n_1^{k_1}\cdots n_r^{k_r}}.
 \]
\end{lem}
We define an automorphism $S_1$ on $\mathfrak{H}$ by $S_1(x):=x$ and $S_1(y):=x+y$, and a $\mathbb{Q}$-linear map $S$ on $y\mathfrak{H}$ by $S(yw) := yS_1(w) \,\, (w\in\mathfrak{H})$.
\begin{lem} \label{thm:z1coeff}
 For positive integers $k_{1},\ldots,k_{r}$, the coefficients of $z^{1}$ in the Taylor expansion of $F_{(k_1,\ldots,k_r)}(z)$ at $z=0$ is given by
 \[
  \mathit{Z} (S (y^{k_1} xy^{k_2-1} \cdots xy^{k_r-1} x)).
 \]
\end{lem}
\begin{proof}
 This is a special case of {\cite[Proposition 3.1]{Yam17}}.
\end{proof}
\begin{lem}[{\cite[Proposition 2.9]{Yam17}}] \label{thm:Kawashima_series}
 For positive integers $r$ and $k_{1},\ldots,k_{r}$,  we have
 \[
  F_{(k_1,\ldots,k_r)}(z)
  =\sum_{n'=1}^\infty \left(\sum_{0<n_1\le\cdots\le n_r=n'}
   \frac{1}{n_1^{k_1} \cdots n_r^{k_r}}-\frac{F_{(k_{1}, \ldots,k_{r-1})}(z+n')}{(z+n')^{k_{r}}} \right).
 \]
\end{lem}
\begin{lem}[Key lemma for Theorem \ref{main3-1}] \label{lem:star_dual}
 For positive integers $k_{1},\ldots,k_{r}$, we have
 \begin{align*}
  &\mathit{Z} (S (y^{k_1} xy^{k_2-1} \cdots xy^{k_r-1} x)) \\
  &=\sum_{j=1}^{r} \sum_{0<n_{1}\le\cdots\le n_{j-1}\le n_{j}>n_{j+1}>\cdots>n_{r}>0}
   (-1)^{r-j} \left(\frac{k_{j}}{n_{j}}\right) \frac{1}{n_1^{k_1}\cdots n_r^{k_r}}.
 \end{align*}
\end{lem}
\begin{proof}
 For non-negative integer $n$, we denote by $A_{(k_1,\ldots,k_r)}^{(j)}(n)$ the coefficients
 of $z^{j}$ in the Taylor expansion of $F_{(k_1,\ldots,k_r)}(z+n)$.
 By replacing $z$ with $z+n$ and seeing the coefficients of $z^{1}$ of Lemma \ref{thm:Kawashima_series}, we have
 \begin{align*}
  &A_{(k_1,\ldots,k_r)}^{(1)}(n) \\
  &=\sum_{n'=1}^\infty \left( \frac{k_r}{(n+n')^{k_r+1}} A_{(k_1,\ldots,k_{r-1})}^{(0)}(n+n')
    - \frac{1}{(n+n')^{k_r}} A_{(k_1,\ldots,k_{r-1})}^{(1)}(n+n') \right) \\
  &=\sum_{n'=1}^\infty \Biggl( \frac{k_r}{(n+n')^{k_r+1}} \sum_{0<n_1\le \cdots \le n_{r-1} \le n+n'}
     \frac{1}{n_1^{k_1}\cdots n_{r-1}^{k_{r-1}}} - \frac{1}{(n+n')^{k_r}} A_{(k_1,\ldots,k_{r-1})}^{(1)}(n+n') \Biggr) \\
  &\qquad\qquad\qquad\qquad\qquad\qquad\qquad\qquad\qquad\qquad\qquad\quad\quad\,\, \textrm{(by Lemma \ref{thm:Fk_at_N})} \\
  &=\sum_{0<n_{1}\le\cdots\le n_{r}>n} \left(\frac{k_{r}}{n_{r}}\right)\frac{1}{n_1^{k_1}\cdots n_r^{k_r}}
     -\sum_{m>n}\frac{1}{m^{k_{r}}}A_{(k_1,\ldots,k_{r-1})}^{(1)}(m).
 \end{align*}
 By using this equation repeatedly and putting $n=0$, we obtain
 \[
  A_{(k_1,\ldots,k_r)}^{(1)}(0)
  =\sum_{j=1}^{r}\sum_{0<n_{1}\le\cdots \le n_{j}>n_{j+1}>\cdots>n_{r}>0}
   \left(\frac{k_{j}}{n_{j}}\right) \frac{(-1)^{r-j}}{n_1^{k_1}\cdots n_r^{k_r}}.
 \]
 Since
 \[
  A_{(k_1,\ldots,k_r)}^{(1)}(0)
  =\mathit{Z} (S (y^{k_1} xy^{k_2-1} \cdots xy^{k_r-1} x))
 \]
 by Lemma \ref{thm:z1coeff}, the claim is proved.
\end{proof}
\begin{thm} \label{main3-1}
 For $A\in y\mathfrak{H}$ and $B\in\mathfrak{H}x$, there exist Dirichlet series $L_{1}(s)$ and $L_{2}(s)$ such that $L_{1}(s)$ and $L_{2}(s)$ converge for all $s\in\mathbb{Z}_{\ge0}$, and
 \[
  \mathit{Z}(A(x+y)^{s}B)=L_{1}(s)+sL_{2}(s)
 \]
for all $s\in\mathbb{Z}_{\ge0}$.
\end{thm}
\begin{proof}
 It is enough to only consider the case $A=y$ and $B=(x+y)^{k_{1}-1}x\cdots(x+y)^{k_{r}-1}x \,\,\,
 (r\in\mathbb{Z}_{\ge0}, k_{1},\ldots,k_{r}\in\mathbb{Z}_{\ge1})$ from Theorem \ref{main1} (i).
 Let $s$ be a non-negative integer.
 We note that
 \[
  \mathit{Z}(A(x+y)^{s}B)
  =\mathit{Z} (S (y^{k_1+s} xy^{k_2-1} \cdots xy^{k_r-1} x)).
 \]
 Thus, by Lemma \ref{lem:star_dual}, we have
 \begin{align*}
  \mathit{Z}(A(x+y)^{s}B)
  &=\sum_{j=1}^{r}\sum_{0<n_{1}\le\cdots\le n_{j}>\cdots>n_{r}>0}
   \left( \frac{k_j}{n_1^s n_j} \right) \frac{(-1)^{r-j}}{n_1^{k_1}\cdots n_r^{k_r}} \\
  &\quad +\sum_{n_{1}>\cdots> n_{r}>0} \left(\frac{s}{n_1^{s+1}}\right)\frac{1}{n_1^{k_1}\cdots n_r^{k_r}}.
 \end{align*}
 Hence, the Dirichlet series
 \[
  L_{1}(s)=\sum_{N=1}^{\infty}\frac{c_{1}(N)}{N^{s}} \quad\textrm{and}\quad L_{2}(s)=\sum_{N=1}^{\infty}\frac{c_{2}(N)}{N^{s}},
 \]
 where
 \begin{align*}
  c_{1}(N)
  &=\sum_{j=1}^{r}\sum_{N=n_{1}\le\cdots\le n_{j}>\cdots>n_{r}>0}
   \left( \frac{k_{j}}{n_{j}} \right) \frac{(-1)^{r-j}}{n_1^{k_1}\cdots n_r^{k_r}}, \\
  c_{2}(N)
  &=\sum_{N=n_{1}>\cdots> n_{r}>0} \left( \frac{1}{n_1} \right) \frac{1}{n_1^{k_1}\cdots n_r^{k_r}}
 \end{align*}
 satisfy the condition.
\end{proof}
\begin{rem}
 By the proof of Theorem \ref{main3-1}, the Dirichlet series $L_{1}(s)$ and $L_{2}(s)$ are explicitly written by the multiple zeta functions.
 Hence, these series can be continued meromorphically to the whole plane $\mathbb{C}$ (for details, see Akiyama-Egami-Tanigawa \cite{AET01} and Zhao \cite{Zha00}).
\end{rem}
\begin{thm} \label{main3-2}
 Let $L_{1}(s) =\sum_{n=1}^{\infty}a_n/n^s$ and $L_{2}(s)=\sum_{n=1}^{\infty}b_n/n^s$ be two Dirichlet series.
 If $L_{1}(s)$ and $L_{2}(s)$ converge for all $s\in\mathbb{Z}_{\ge0}$, and $L_{1}(s)+sL_{2}(s)=0$ for all $s\in\mathbb{Z}_{\ge0}$,
 then we have $a_i=b_i=0$ for all $i$.
\end{thm}
\begin{proof}
 If the claim in the theorem does not hold, there exists a positive integer $N$ such that
 \begin{align*}
  L_1(s)+sL_2(s)=\sum_{n=N}^{\infty}\frac{a_n+sb_n}{n^s},
 \end{align*}
 where $a_N+sb_N\neq0$.
 Here, we split the sum as
 \begin{align*}
 \sum_{n=N}^{\infty}\frac{a_n+sb_n}{n^s}=I_1(s)+I_2(s),
 \end{align*}
 where
 \begin{align*}
 I_1(s)=\frac{a_N+sb_N}{N^s} \quad\textrm{and}\quad I_2(s) = \sum_{n=N+1}^{\infty}\frac{a_n+sb_n}{n^s}.
 \end{align*}
 Since $L_1(s)$ and $L_2(s)$ are convergent at $s=0$,
 $a_n=O(1)$ and $b_n=O(1)$ hold.
 Hence for large real $s$, we have
 \begin{align*}
  \left|\sum_{n=N+2}^{\infty}\frac{a_n+sb_n}{n^s}\right|\ll|s|\int_{N+1}^\infty x^{-s}dx\ll(N+1)^{-s}.
 \end{align*}
 Therefore we have
 \begin{align*}
  I_2(s)=\frac{a_{N+1}+sb_{N+1}}{(N+1)^s}+\sum_{n=N+2}^{\infty}\frac{a_n+sb_n}{n^s}=O\left(s(N+1)^{-s}\right).
 \end{align*}
 Hence there exists $\sigma_0\in\mathbb{R}$ such that $|I_1(s)|>|I_2(s)|$ holds for any real $s>\sigma_0$.
 On the other hand, we have $I_1(s)+I_2(s)=0$ for any non-negative integer $s$.
 This is a contradiction, and the theorem is proved.
\end{proof}

\begin{proof}[Proof of Theorem \ref{main1} (iii)]
 By Theorems \ref{main3-1} and \ref{main3-2}, we have the desired result.
\end{proof}

Now, we prove Theorem \ref{main1} (ii).
We need the following lemma (for a proof, see e.g., Panzer \cite[Lemma 3.3.5]{Pan15}).
\begin{lem} \label{poincare}
The set of complex functions
\[
\{{\rm Li}_{\Bbbk}(z)\mid r\in \mathbb{Z}_{\geq0},\Bbbk\in\mathbb{Z}_{>0}^{r}\}
\]
on $\{|z|<1\}$ is $\mathbb{C}$-linearly independent, where ${\rm Li}_{\Bbbk}(z)$ are the multiple polylogarithms defined by
\[
{\rm Li}_{k_1,\dots,k_r}(z)=\sum_{0<n_{1}<\cdots<n_{r}}\frac{z^{n_{r}}}{n_{1}^{k_{1}}\cdots n_{r}^{k_{r}}}.
\]
\end{lem}

\begin{proof}[Proof of Theorem \ref{main1} (ii)]
 Take $w\in\mathfrak{H}x$ such that $f(y,w)=0$.
 For $N\in\mathbb{Z}_{>0}$, we define a linear map $h_{N}:y\mathfrak{H}\to\mathbb{Q}$
by
\[
h_{N}(yx^{k_{1}-1}\cdots yx^{k_{r}-1})=\sum_{0<n_{1}<\cdots<n_{r}=N}\frac{1}{n_{1}^{k_{1}}\cdots n_{r}^{k_{r}}}.
\]
By the proof of Theorem \ref{main3-1}, there exist two Dirichlet series
\[
L_{1}(w;s)=\sum_{N=1}^{\infty}\frac{c_{1}(w;N)}{N^{s}} \quad\textrm{and}\quad L_{2}(w;s)=\sum_{N=1}^{\infty}\frac{c_{2}(w;N)}{N^{s}}
\]
such that $L_{1}(w;s)$ and $L_{2}(w;s)$ converge for all $s\in\mathbb{Z}_{\geq0}$,
\[
c_{2}(w;N)=\frac{h_{N}(\tau\circ S_{1}^{-1}(w))}{N}
\]
for all $N\in\mathbb{Z}_{>0}$, and
\[
L_{1}(w;s)+sL_{2}(w;s)=0
\]
for all $s\in\mathbb{Z}_{\geq0}$. Furthermore, by Theorem \ref{main3-2}, $c_{2}(w;N)=0$
for all $N\in\mathbb{Z}_{>0}$. Thus
\begin{align*}
 0 & =\sum_{N>0}Nc_{2}(w;N)z^{N}\\
   & =\sum_{r>0}\sum_{\Bbbk\in\mathbb{Z}_{>0}^{r}}n(\Bbbk){\rm Li}_{\Bbbk}(z),
\end{align*}
where $n(k_{1},\dots,k_{r})$ are coefficients of $yx^{k_{1}-1}\cdots yx^{k_{r}-1}$
in $\tau\circ S_{1}^{-1}(w)$. Therefore by Lemma \ref{poincare}, $n(\Bbbk)=0$
for all $\Bbbk$. Now Theorem \ref{main1} (ii) is proved since
\[
w=S_{1}\circ\tau\left(\sum_{\Bbbk}n(\Bbbk)yx^{k_{1}-1}\cdots yx^{k_{r}-1}\right)=0.\qedhere
\]
\end{proof}

\section{Some equivalences of families of relations}
In this section, we prove Theorem \ref{Thm:equivalence} which gives
characterizations of the linear part of Kawashima's relation.
Roughly speaking, $A_1$ is the set of ($\mathbb{Q}$-linear sums of)
the linear part of Kawashima's relation, and $A_5,A_6$ are sets of $\mathbb{Q}$-linear
relations which are extendable to $\mathbb{Q}$-linear relations
between families $\{f(A,B)\mid A\in y\mathfrak{H},B\in\mathfrak{H}x\}$.
Theorem \ref{rem1} is a special case $A_6 \subset A_1$ of Theorem \ref{Thm:equivalence}.
%
\begin{thm} \label{Thm:equivalence}
 The following vector subspaces $A_{1},\ldots,A_{6}$ of $\mathfrak{H}$ are coincident.
 \begin{align*}
  A_{1}&:=\sspan\{\phi(u\ast v)x\mid u,v\in y\mathfrak{H}\},\\
  A_{2}&:=\sspan\{(u\diamond v)x\mid u,v\in y\mathfrak{H}\},\\
  A_{3}&:=\sspan\{yw_{1}(x+y)^{n}w_{2}x-y(x+y)^{n}(\tau(w_{1})\diamond w_{2})x\mid w_{1},w_{2}\in\mathfrak{H},n\in\mathbb{Z}_{\ge0}\},\\
  A_{4}&:=\sspan\{yw_{1}w_{2}x-y(\tau(w_{1})\diamond w_{2})x\mid w_{1},w_{2}\in\mathfrak{H}\},\\
  A_{5}&:=\sspan \biggl\{ \sum_{i=1}^{M} n_{i}A_{i}(x+y)^{s}B_{i} \;\biggl|\;
               s,M\in\mathbb{Z}_{\ge0} ,n_{i}\in\mathbb{Q},A_{i}\in y\mathfrak{H},B_{i}\in\mathfrak{H}x, \\
        &\qquad\qquad\qquad\qquad\qquad\qquad\qquad\qquad\qquad\qquad \sum_{i=1}^{M}n_{i}f(A_{i},B_{i})=0 \biggr\}, \\
  A_{6}&:=\sspan \biggl\{ \sum_{i=1}^{M} n_{i}A_{i} B_{i} \;\biggl|\;
               M\in\mathbb{Z}_{\ge0} ,n_{i}\in\mathbb{Q},A_{i}\in y\mathfrak{H},B_{i}\in\mathfrak{H}x, \\
        &\qquad\qquad\qquad\qquad\qquad\qquad\qquad\qquad\qquad\qquad \sum_{i=1}^{M}n_{i}f(A_{i},B_{i})=0 \biggr\}.
 \end{align*}
\end{thm}
\begin{proof}
 The equality of $A_{1}$, $A_{2}$, $A_{3}$, and $A_{4}$ was proved in Lemma \ref{key1}.
 The inclusion relation $A_{4}\subset A_{6}\subset A_{5}$ is obvious by Theorem \ref{main1} (i).
 Thus, we need to show $A_{5}\subset A_{3}$.
 Take $M\in\mathbb{Z}_{\ge0},n_{i}\in\mathbb{Q},A_{i}\in y\mathfrak{H}$, and $B_{i}\in\mathfrak{H}x$ such that
 \[
  \sum_{i=1}^{M}n_{i}f(A_{i},B_{i})=0.
 \]
 Put $yw_{1}^{(i)}:=A_{i}$ and $w_{2}^{(i)}x:=B_{i}$.
 Then, by Theorem \ref{main1} (i), we have
 \[
  \sum_{i=1}^{M} n_{i}f(y,(\tau(w_{1}^{(i)})\diamond w_{2}^{(i)})x)=0.
 \]
 Thus, by Theorem \ref{main1} (ii), we have
 \[
  \sum_{i=1}^{M}n_{i}\tau(w_{1}^{(i)})\diamond w_{2}^{(i)}=0.
 \]
 Therefore, we get
 \begin{align*}
  &\sum_{i=1}^{M} n_{i}A_{i}(x+y)^{s}B_{i} \\
  &=\sum_{i=1}^{M} n_{i}\left(yw_{1}^{(i)}(x+y)^{s}w_{2}^{(i)}x-y(x+y)^{s}(\tau(w_{1}^{(i)})\diamond w_{2}^{(i)})x\right)
   \in A_{3}. \qedhere
 \end{align*}
\end{proof}

\section*{Acknowledgements}
The authors would like to thank Professor Masanobu Kaneko for valuable comments.
This work was supported by JSPS KAKENHI Grant Numbers JP18J00982, JP18K13392.


\end{document}